\documentclass[12pt]{amsart}
\pdfoutput=1

\usepackage[paperheight=11in, paperwidth=8.5in, left=1in, top=1.25in, right=1in, bottom=1.25in]{geometry}
\usepackage[linktocpage=true, linktoc=all, colorlinks=true, linkcolor=black, citecolor=black, filecolor=black, urlcolor=black, pagebackref=false, pdfstartpage={1}, pdfstartview={FitH}, pdftitle={}, pdfauthor={}, pdfsubject={}, pdfcreator={}, pdfproducer={}, pdfkeywords={}]{hyperref}
\usepackage{afterpage}
\usepackage{amsfonts}
\usepackage{amsmath}
\allowdisplaybreaks[4]
\usepackage{amssymb}
\usepackage{amsthm}
\usepackage{array}
\usepackage{arydshln}
\usepackage{bbm}
\usepackage[justification=centering]{caption}
\usepackage[capitalize,nameinlink,noabbrev,compress]{cleveref}

\usepackage{enumerate}
\usepackage{enumitem}
\usepackage{float}
\restylefloat{figure}
\usepackage{graphicx}
\usepackage{kbordermatrix}
\renewcommand{\kbldelim}{(}
\renewcommand{\kbrdelim}{)}
\usepackage{mathrsfs}
\usepackage[framemethod=tikz,ntheorem,xcolor]{mdframed}
\usepackage{multirow}
\usepackage{parcolumns}
\usepackage{scalerel}
\usepackage{tikz}\pdfpageattr{/Group <</S /Transparency /I true /CS /DeviceRGB>>}
\usetikzlibrary{arrows, calc, cd, decorations.markings, intersections, matrix, positioning, through}
\tikzset{commutative diagrams/.cd, every label/.append style = {font = \normalsize}}
\usepackage[all]{xy}
\usepackage[aligntableaux=center]{ytableau}

\numberwithin{equation}{section}
\newtheorem{theorem}{Theorem}
\AfterEndEnvironment{theorem}{\noindent\ignorespaces}
\numberwithin{theorem}{section}

\AfterEndEnvironment{conjecture}{\noindent\ignorespaces}

\AfterEndEnvironment{corollary}{\noindent\ignorespaces}
\newtheorem{lemma}[theorem]{Lemma}
\AfterEndEnvironment{lemma}{\noindent\ignorespaces}

\AfterEndEnvironment{problem}{\noindent\ignorespaces}
\newtheorem{proposition}[theorem]{Proposition}
\AfterEndEnvironment{proposition}{\noindent\ignorespaces}

\theoremstyle{definition}

\AfterEndEnvironment{definition}{\noindent\ignorespaces}
\newtheorem{example_no_qed}[theorem]{Example}
\newenvironment{example}[1][]{\begin{example_no_qed}[#1]\pushQED{\qed}}{\popQED\end{example_no_qed}}
\AfterEndEnvironment{example}{\noindent\ignorespaces}
\newtheorem{remark}[theorem]{Remark}
\AfterEndEnvironment{remark}{\noindent\ignorespaces}

\theoremstyle{remark}

\AfterEndEnvironment{claim}{\noindent\ignorespaces}
\newtheorem*{claim_proof_no_qed}{Proof of Claim}

\AfterEndEnvironment{claim_proof}{\noindent\ignorespaces}

\AddToHook{env/proposition/begin}{\crefalias{theorem}{proposition}}
\AddToHook{env/conjecture/begin}{\crefalias{theorem}{conjecture}}
\AddToHook{env/corollary/begin}{\crefalias{theorem}{corollary}}
\AddToHook{env/lemma/begin}{\crefalias{theorem}{lemma}}
\AddToHook{env/problem/begin}{\crefalias{theorem}{problem}}
\AddToHook{env/proposition/begin}{\crefalias{theorem}{proposition}}
\AddToHook{env/definition/begin}{\crefalias{theorem}{definition}}
\AddToHook{env/example/begin}{\crefalias{theorem}{example}}
\AddToHook{env/example_no_qed/begin}{\crefalias{theorem}{example_no_qed}}
\AddToHook{env/remark/begin}{\crefalias{theorem}{remark}}

\crefname{theorem}{Theorem}{Theorems}
\crefname{conjecture}{Conjecture}{Conjectures}
\crefname{corollary}{Corollary}{Corollaries}
\crefname{lemma}{Lemma}{Lemmas}
\crefname{problem}{Problem}{Problems}
\crefname{proposition}{Proposition}{Propositions}
\crefname{definition}{Definition}{Definitions}
\crefname{example}{Example}{Examples}
\crefname{example_no_qed}{Example}{Examples}
\crefname{remark}{Remark}{Remarks}

\newcommand{\define}[1]{{\bfseries\itshape #1}}

\DeclareMathOperator{\End}{End}
\renewcommand{\eqref}[1]{\hyperref[#1]{\textup{(\ref*{#1})}}}

\DeclareMathOperator{\Hom}{Hom}

\DeclareMathOperator{\interior}{int}

\DeclareMathOperator{\PSL}{PSL}

\DeclareMathOperator{\SL}{SL}

\newcommand{\transpose}[1]{{#1}{\hspace*{-0.2pt}\raisebox{4pt}{$\scriptstyle\mathsf{T}$}\hspace*{0.5pt}}}
\newcommand{\transposeminus}[1]{{#1}{\hspace*{-4.8pt}\raisebox{4pt}{$\scriptstyle\mathsf{T}$}\hspace*{0.5pt}}}

\newcommand{\att}{\textnormal{att}}
\newcommand{\symgp}[1]{\mathfrak{S}_{#1}}

\title{The totally nonnegative wonderful compactification of \texorpdfstring{$G$}{G} is a ball}

\author{Sam Evens}
\address{Department of Mathematics, University of Notre Dame}
\email{\href{mailto:sevens@nd.edu}{sevens@nd.edu}}
\author{Steven N. Karp}
\address{Department of Mathematics, University of Notre Dame}
\email{\href{mailto:skarp2@nd.edu}{skarp2@nd.edu}}


\begin{document}

\begin{abstract}
Let $\overline{G}$ denote the wonderful compactification of a semisimple algebraic group $G$ of adjoint type. We use contractive flows to show that the totally nonnegative part of every $G\times G$ orbit closure of $\overline{G}$ is homeomorphic to a closed Euclidean ball. In particular, the totally nonnegative part of $\overline{G}$ is homeomorphic to a closed ball.
\end{abstract}

\maketitle

\section{Introduction}\label{sec:intro}

\noindent Let $G$ denote a complex connected semisimple algebraic group of adjoint type. De Concini and Procesi \cite{de_concini_procesi83} introduced the \emph{wonderful compactification} $\overline{G}$ of $G$ (see \cref{sec:Gbar} for the definition). It has since been extensively studied, appearing in representation theory \cite{bezrukavnikov_kazhdan15}, arithmetic geometry and ergodic theory \cite{gorodnik_maucourant_oh08}, Poisson geometry \cite{evens_lu01}, and the moduli of principal bundles \cite{martens_thaddeus16}. The group $G\times G$ acts on $\overline{G}$, and its orbits $Z_J$ are indexed by subsets $J$ of the set of nodes $I$ of the Dynkin diagram of $G$, where $Z_I = G$. The compactification $\overline{G}$ is called `wonderful' because not only is $\overline{G}$ smooth, but so is every orbit closure $\overline{Z_J}$.

On the other hand, inspired by both the classical theory of total positivity for matrices \cite{karlin68} and the representation theory of canonical bases for quantum groups \cite{lusztig90}, Lusztig introduced the totally positive and totally nonnegative parts of semisimple algebraic groups \cite{lusztig94}. Motivated by this work, He \cite{he04} studied the totally nonnegative part $\overline{G_{\ge 0}}$ of the wonderful compactification $\overline{G}$. He constructed a cell decomposition of $\overline{G_{\ge 0}}$ (resolving a problem of Lusztig \cite[Section 9]{lusztig94}), and showed that $\overline{G_{\ge 0}}$ and the totally nonnegative part $Z_J^{\ge 0}$ of every $G\times G$ orbit are contractible.

Our main result determines the topology of the totally nonnegative part of every $G\times G$ orbit closure of $\overline{G}$:
\begin{theorem}\label{main}
For all $J\subseteq I$, let $\overline{Z_J^{\ge 0}}$ denote the totally nonnegative part of the $G\times G$ orbit closure of $\overline{G}$ indexed by $J$. Then $\overline{Z_J^{\ge 0}}$ is homeomorphic to a closed Euclidean ball. In particular, $\overline{G_{\ge 0}}$ is homeomorphic to a closed ball.
\end{theorem}

\subsection{Motivation}\label{sec:motivation}
We provide some background motivation for studying totally positive spaces and their topology. Since Lusztig's work on total positivity and the development of cluster algebras by Fomin and Zelevinsky \cite{fomin_zelevinsky02} in the 1990's, totally positive spaces have attracted a lot of attention and found diverse applications. For example, they appear in work in soliton solutions of the KP equation \cite{kodama_williams14}, scattering amplitudes in physics \cite{arkani-hamed_bourjaily_cachazo_goncharov_postnikov_trnka16,williams23}, higher Teichm\"{u}ller theory \cite{guichard_wienhard18}, the Ising model \cite{galashin_pylyavskyy20}, and real algebraic geometry \cite{karp_purbhoo}.

The first motivation for studying the topology of totally positive spaces dates to work of Bj\"{o}rner \cite{bjorner84} in the 1980's on regular CW complexes, which are topological spaces with a cell decomposition in which every cell closure is homeomorphic to a closed ball. A regular CW complex is a combinatorial object in the sense that it is uniquely determined by its closure poset. Total positivity turns out to provide a rich source of regular CW complexes with combinatorially interesting face posets. For example, the link of the identity in the totally nonnegative part of the unipotent radical $U_+$ of the Borel subgroup $B_+\subseteq G$ is a regular CW complex whose face poset is the Weyl group $W$ under the Bruhat order \cite{fomin_shapiro00,hersh14} (answering a question of Bernstein and Bj\"{o}rner), and the totally nonnegative part of the flag variety $G/B_+$ is a regular CW complex whose face poset is the poset of intervals of $W$ under containment \cite{galashin_karp_lam22}.

There are many more instances of this phenomenon which are conjectured \cite[Section 10]{galashin_karp_lam22} (see \cite{bao_he24,bao_hec} for some recent progress in this area). In particular, \cref{main} verifies a special case of the conjecture from \cite[Section 10]{galashin_karp_lam22} that He's cell decomposition of $\overline{G_{\ge 0}}$ is a regular CW complex. While we were writing our paper, He and Xie \cite{he_xie} posted a proof of this conjecture using different methods; see \cref{he_xie_remark}.

A second motivation comes from the recent development of \emph{positive geometries} \cite{arkani-hamed_bai_lam17}, which uses totally positive spaces and their associated canonical differential forms in order to study scattering amplitudes and wavefunctions in physics. For example, a certain projection of the totally nonnegative Grassmannian called an \emph{amplituhedron} encodes tree-level scattering amplitudes in planar $\mathcal{N} = 4$ supersymmetric Yang--Mills theory \cite{arkani-hamed_trnka14}. In this context, cell decompositions of the associated totally positive space can encode different expressions for the scattering amplitude as well as its factorization properties. This provides further impetus for studying the combinatorics and topology of totally positive spaces.

\subsection{Proof technique}\label{sec:proofs}
In order to prove \cref{main} we construct a \emph{contractive flow} on $\overline{Z_J^{\ge 0}}$, which contracts the entire space to a unique attractor. This technique was introduced in \cite{galashin_karp_lam22a} in order to show that the totally nonnegative Grassmannian is homeomorphic to a closed ball. It has since been applied to show that many other spaces appearing in total positivity \cite{galashin_karp_lam19,galashin_pylyavskyy20,fraser,bloch_karp23a,shevchenko,guichard_wienhard} and in the theory of Lorentzian polynomials \cite{branden21,chin_qin} are homeomorphic to closed balls.

There is a natural candidate for a contractive flow on $\overline{Z_J^{\ge 0}}$, which essentially appeared in \cite{lusztig98a,he04,galashin_karp_lam19}: the flow generated by the sum $\sum_{i\in I}(e_i + f_i)$ of all Chevalley generators in the Lie algebra of $G$. The main obstacle in using this flow in order to show that $\overline{Z_J^{\ge 0}}$ is a closed ball is that, unlike all of the applications of contractive flows listed above, the attractor of the flow is on the boundary of the space rather than in its interior. To overcome this obstacle, we do two things:
\begin{enumerate}[label=(\arabic*), leftmargin=36pt, itemsep=4pt]
\item We prove a general topological result (\cref{flow_homeomorphism}) which applies to contractive flows whose attractor lies on the boundary. It essentially states that the homeomorphism type of the entire space is determined by the local topology around its attractor. The statement and proof only use elementary point-set topology. We are hopeful that \cref{flow_homeomorphism} will expand the range of spaces to which the technique of contractive flows can be applied in the future.
\item We determine the local topology around the attractor of the contractive flow on $\overline{Z_J^{\ge 0}}$. Specifically, we show that the attractor is contained in the smallest $G\times G$ orbit $Z_\varnothing$, and that the boundary of $\overline{Z_J^{\ge 0}}$ near the attractor locally looks like a vertex of a $|J|$-dimensional cube times several copies of $\mathbb{R}_{>0}$. In particular, the topology near the attractor is the same as that of a point on the boundary of a closed ball. This is enough to deduce that $\overline{Z_J^{\ge 0}}$ is homeomorphic to a closed ball using \cref{flow_homeomorphism}.
\end{enumerate}

\begin{remark}\label{he_xie_remark}
As we mentioned above, while we were writing up our results, He and Xie \cite{he_xie} posted a proof that $\overline{G_{\ge 0}}$ forms a regular CW complex. This is certainly stronger than \cref{main} since it implies that every cell closure of $\overline{G_{\ge 0}}$ is homeomorphic to a closed ball. (While each $\overline{Z_J^{\ge 0}}$ is an instance of such a cell closure, there are many others.) However, the two approaches used are quite different: while we use contractive flows, \cite{he_xie} uses the atlas method. We briefly compare these two techniques, each of which has its pros and cons.

The atlas method was first used in \cite{galashin_karp_lam22} to prove that totally nonnegative flag varieties are regular CW complexes, building on work of Fomin and Shapiro \cite{fomin_shapiro00} and He, Knutson, and Lu (unpublished). It is powerful and has (so far) proved stronger results than contractive flows. However, it is quite technical and the underlying homeomorphisms rely on the generalized Poincar\'{e} conjecture. On the other hand, contractive flows are more elementary and provide explicit homeomorphisms. This leads to convenient coordinates in which to plot these spaces (for example, see \cref{fig:PSL(2)}). Also, it is not necessary to construct a cell decomposition of a space in order to use contractive flows.
\end{remark}

\subsection{Outline of the paper}\label{sec:outline}
In \cref{sec:background}, we recall the necessary background on algebraic groups and the wonderful compactification of $G$. In \cref{sec:contracting}, we recall the definition of contractive flows and prove our general topological result (\cref{flow_homeomorphism}). In \cref{sec:flow} we define the contractive flow on $\overline{Z_J^{\ge 0}}$ and prove several of its properties, culminating in the proof of \cref{main} in \cref{sec:flow_proof}.

\subsection*{AI usage}
The authors did not use generative AI at any stage of this project.

\subsection*{Acknowledgments}
S.N.K.\ was partially supported by the National Science Foundation under Award No.\ 2452061 and by a travel support gift from the Simons Foundation.

\section{Background}\label{sec:background}

\noindent We recall some background on algebraic groups \cite{humphreys75}, total positivity \cite{lusztig94}, and the wonderful compactification \cite{de_concini_procesi83,he04,evens_jones08}. Some of our exposition follows \cite{galashin_karp_lam22,barkley_karp}.

\subsection{Algebraic groups}\label{sec:background_groups}
Let $G$ be a complex connected semisimple algebraic group of adjoint type which is split over $\mathbb{R}$. Fix an $\mathbb{R}$-split maximal torus $T$ of $G$. Let $B_+$ and $B_-$ be Borel subgroups of $G$ such that $T = B_+ \cap B_-$, and denote their unipotent radicals by $U_+$ and $U_-$, respectively. Let $X(T) := \Hom(T,\mathbb{C}^\times)$ denote the \define{weight lattice}, and let $\{\alpha_i : i\in I\} \subseteq X(T)$ denote the set of \define{simple roots} induced by our choice of $B_+$ and $B_-$.

For every $i\in I$ we pick a homomorphism $\phi_i: \SL_2(\mathbb{C}) \to G$, and define
\[
x_i(t) := \phi_i\!\left(\begin{matrix}1 & t \\ 0 & 1\end{matrix}\right) \quad \text{ and } \quad
y_i(t) := \phi_i\!\left(\begin{matrix}1 & 0 \\ t & 1\end{matrix}\right).
\]
We require that $x_i(\mathbb{C})$ is the root subgroup of $U_+$ corresponding to $\alpha_i$, and similarly that $y_i(\mathbb{C})$ is the root subgroup of $U_-$ corresponding to $-\alpha_i$. We call the data $(T, B_+, B_-, I, \{x_i : i\in I\}, \{y_i :  i\in I\})$ a \define{pinning} for $G$. We also let $\transpose{\cdot}$ denote the unique involutive anti-automorphism of $G$ which satisfies
\begin{align}\label{lie_transpose}
\transpose{g} = g \text{ for all } g\in T
\quad \text{ and } \quad
\transpose{x_i(t)} = y_i(t) \text{ for all $i\in I$ and $t\in\mathbb{C}$}.
\end{align}

We let $W := N(T)/T$ denote the \define{Weyl group} of $G$, with simple generators $s_i$ for $i\in I$ and longest element $w_0\in W$. We define the \define{coweight lattice} $Y(T) := \Hom(\mathbb{C}^\times, T)$, and for all $i\in I$, we have the coroot $\alpha^\vee_i$ defined by
\[
\alpha^\vee_i(t) := \phi_i\!\left(\begin{matrix}t & 0 \\ 0 & t^{-1}\end{matrix}\right) \quad \text{ for all } t\in\mathbb{C}^\times.
\]
We let $\mathfrak{g}$ denote the Lie algebra of $G$, with Chevalley generators $e_i$, $f_i$, and $h_i := [e_i,f_i]$ for $i\in I$, where $x_i(1) = \exp(e_i)$ and $y_i(1) = \exp(f_i)$. Let $\mathfrak{t}$ denote the Lie algebra of $T$.

\begin{example}\label{eg:pinning}
We recall the standard pinning for $G = \PSL_n(\mathbb{C})$, the group of invertible $n\times n$ matrices with determinant $1$ modulo its center (i.e.\ modulo rescaling by $n$th roots of unity). Modulo the center, we let $B_+$ be the subgroup of upper-triangular matrices and $B_-$ be the subgroup of lower-triangular matrices, so that $T = B_+ \cap B_-$ is the subgroup of diagonal matrices. We have $I = \{1, \dots, n-1\}$, and for $i\in I$ the map $\phi_i$ embeds $\SL_2(\mathbb{C})$ in rows and columns $\{i,i+1\}$ (except when $n=2$, in which case $\phi_1$ is a $2$-to-$1$ cover). For example, modulo the center,
\[
x_i(t) = \kbordermatrix{& & i & i+1 & \cr
& \ddots & & & \cr
i & & 1 & t & \cr
i+1 & & 0 & 1 & \cr
& & & & \ddots}
\quad\text{ and }\quad
\alpha^\vee_i(t) = \kbordermatrix{& & i & i+1 & \cr
& \ddots & & & \cr
i & & t & 0 & \cr
i+1 & & 0 & t^{-1} & \cr
& & & & \ddots},
\]
where the matrices agree with the identity matrix $I_n$ in all unspecified entries. The map $\transpose{\cdot}$ is the usual matrix transpose. The Weyl group $W$ is the symmetric group $\symgp{n}$ of all permutations of $\{1, \dots, n\}$.
\end{example}

\subsection{Total positivity}\label{sec:positivity}
Following \cite{lusztig94}, we introduce the \define{totally nonnegative} and \define{totally positive parts} of various spaces introduced above, which we denote by adding `$\ge\!0$' and `$>\!0$', respectively, as a superscript or subscript. We define $U_+^{\ge 0}$ to be the submonoid of $U_+$ generated by $x_i(t)$ for all $i\in I$ and $t\in\mathbb{R}_{>0}$, and we define
\begin{align}\label{Utp_definition}
U_+^{>0} := U_+^{\ge 0} \cap B_-nB_-,
\end{align}
where $n\in N(T)$ denotes any representative of $w_0\in W$. We define $U_-^{\ge 0} = \transpose{(U_+^{\ge 0})}$ and $U_-^{>0} = \transpose{(U_+^{>0})}$.

We define $T_{>0}$ to be the submonoid of $T$ generated by $\alpha^\vee_i(t)$ for all $i\in I$ and $t\in\mathbb{R}_{>0}$. We define $G_{\ge 0}$ to be the submonoid of $G$ generated by $U_+^{\ge 0}$, $U_-^{\ge 0}$, and $T_{>0}$, and we define
\[
G_{>0} := U_-^{>0} T_{>0} U_+^{>0} = U_+^{>0} T_{>0} U_-^{>0}.
\]

\subsection{Wonderful compactification}\label{sec:Gbar}
Given a dominant weight $\lambda\in X(T)$, let $V_\lambda$ be the corresponding irreducible representation of $G$ with highest weight $\lambda$. We let $v_0, \dots, v_m\in V_\lambda$ denote a weight basis with weights $\mu_0, \dots, \mu_m$, such that $v_0$ is the unique basis vector of highest weight $\lambda$. Then $\mu_0 = \lambda$, and for $1 \le j \le m$ we write
\begin{align}\label{weight_basis_expansion}
\lambda-\mu_j = \sum_{i\in I}c_{i,j}\alpha_i,
\end{align}
where $c_{i,j}\in\mathbb{N}$ and at least one value of $c_{i,j}$ (over all $i\in I$, for a fixed $j$) is nonzero.

Let $v_0^*, \dots, v_m^*\in V_\lambda^*$ denote the dual basis of $V_\lambda$. We identify $\End(V_\lambda)$ with $V_\lambda\otimes V_\lambda^*$, which has the basis $v_i\otimes v_j^*$ (for $0 \le i,j \le m$). We let $G\times G$ act on $\End(V_\lambda)$ by $(g_1, g_2)\cdot A := g_1Ag_2^{-1}$ for all $A\in\End(V_\lambda)$. Equivalently,
\begin{align}\label{GG_action_tensor}
(g_1, g_2) \cdot v_i\otimes v_j^* := g_1v_i \otimes g_2v_j^* \quad \text{ for all } 0 \le i,j \le m.
\end{align}

Now suppose that $\lambda$ is a dominant regular weight. Then $G$ embeds diagonally into $\mathbb{P}(\End(V_\lambda))$, and the \define{wonderful compactification} $\overline{G}$ of $G$ is defined to be the Zariski closure of $G$ inside $\mathbb{P}(\End(V_\lambda))$.

The $G\times G$-orbits of $\overline{G}$ are indexed by subsets $J\subseteq I$. Explicitly, following \cite[Sections 2.1 and 2.6]{he04}, $Z_J$ is defined to be the $G\times G$ orbit of the distinguished point $z_J^\circ\in\overline{T}$, which is the limit of $h\in T$ such that $\alpha_i(h) = 1$ for all $i\in J$ and $\alpha_i(h)\to\infty$ for all $i\notin J$. (So for example, $z_I^\circ = e$ and $Z_I\cong G$.) We will need the result about the Zariski closure $\overline{Z_J}$ of $Z_J$:
\begin{lemma}[{\cite[Theorem 3.1(iii)]{de_concini_procesi83}}]\label{ZJ_closure}
For all $J\subseteq I$, we have $\overline{Z_J} = \bigsqcup_{J'\subseteq J}Z_{J'}$.\hfill\qed
\end{lemma}

We let $\overline{G_{\ge 0}}$ denote the Euclidean closure of $G_{\ge 0}$ inside $\overline{G}$. (Since $G_{\ge 0}$ is the Euclidean closure of $G_{>0}$ inside $G$, we see that $\overline{G_{\ge 0}}$ equals the Euclidean closure of $G_{>0}$ inside $\overline{G}$, so we could have denoted $\overline{G_{\ge 0}}$ by $\overline{G_{>0}}$.) We similarly let $\overline{T_{>0}}$ denote the Euclidean closure of $T_{>0}$ inside $\overline{G}$. We let $Z_J^{\ge 0} := Z_J \cap \overline{G_{\ge 0}}$ denote the totally nonnegative part of $Z_J$, and let $\overline{Z_J^{\ge 0}}$ denote its Euclidean closure. (We will show in \cref{stratum_closure_description} that $\overline{Z_J^{\ge 0}} = \overline{Z_J}\cap\overline{G_{\ge 0}}$.) We have $z_J^\circ \in Z_J^{\ge 0}$. Following \cite[Section 2.6]{he04}, we also define
\[
Z_J^{>0} := \{(g_1,g_2^{-1})\cdot z_J^\circ : g_1,g_2\in G_{>0}\}.
\]
We will need the following properties of $Z_J^{>0}$ and $Z_J^{\ge 0}$:
\begin{lemma}[{He \cite[Section 2.6 and Corollary 2.9]{he04}}]\label{he_ZJ}
Let $J\subseteq I$.
\begin{enumerate}[label=(\roman*), leftmargin=*, itemsep=2pt]
\item\label{he_ZJ_paramtrization} We have
\[
Z_J^{>0} = \{(u_-h,u_+^{-1})\cdot z_J^\circ : u_-\in U_-^{>0}, h\in T_{>0}, u_+\in U_+^{>0}\}.
\]
Moreover, $u_-$ and $u_+$ above (in the expression of an element of $Z_J^{>0}$) are unique.
\item\label{he_ZJ_containment} For all $g_1,g_2\in G_{>0}$, we have $(g_1,g_2^{-1})\cdot Z_J^{\ge 0} \subseteq Z_J^{>0}$.\hfill\qed
\end{enumerate}
\end{lemma}

\subsection{Folding}\label{sec:folding}
We say that $G$ is \define{simply-laced} if its Cartan matrix is symmetric. The well-known technique of \emph{folding} (also called \emph{descent}) allows one to generalize results for simply-laced $G$ to all $G$. The reason for using folding in total positivity is that desirable positivity properties of the canonical basis from \cite{lusztig90} only hold in the simply-laced case (see \cite[Section 2.7]{barkley_karp} for a more detailed discussion). We will use folding to prove \cref{nonnegative_to_positive}.

Our setup follows \cite[Section 2.8]{he04}. Given $G$ as in \cref{sec:background_groups}, there exists a complex connected semisimple algebraic group $\dot{G}$ of adjoint type which is split over $\mathbb{R}$ and is simply-laced, along with a pinning $(\dot{T}, \dot{B}_+, \dot{B}_-, \dot{I}, \{\dot{x}_i : i\in \dot{I}\}, \{\dot{y}_i :  i\in \dot{I}\})$ and an automorphism $\sigma : \dot{G} \to \dot{G}$ defined over $\mathbb{R}$, satisfying the conditions below. (We point out that we allow $G$ to be simply-laced, in which case we can take $\dot{G} = G$ and $\sigma$ to be the identity.)
\begin{itemize}[itemsep=2pt]
\item The automorphism $\sigma$ preserves the pinning of $\dot{G}$. That is, $\sigma$ stabilizes $\dot{T}$, $\dot{B}_+$, and $\dot{B}_-$ setwise, and $\sigma$ induces a bijection on $\dot{I}$ such that
\[
\sigma(\dot{x}_i(t)) = \dot{x}_{\sigma(i)}(t) \quad \text{ and } \quad \sigma(\dot{y}_i(t)) = \dot{y}_{\sigma(i)}(t) \quad \text{ for all $i\in\dot{I}$ and $t\in\mathbb{C}$}.
\]
\item If $i$ and $j$ are in the same orbit of $\sigma$ on $\dot{I}$, then $i$ and $j$ are in the same connected component of the Dynkin diagram but $\{i,j\}$ is not an edge.
\item We have $G = \dot{G}^\sigma$, $T = \dot{T}^\sigma$, $B_+ = \dot{B}_+^\sigma$, $B_- = \dot{B}_-^\sigma$, and $W = \dot{W}^\sigma$, where $\cdot^\sigma$ denotes taking the fixed-point set of $\sigma$. We can identify $I$ with the set of $\sigma$-orbits of $\dot{I}$, such that
\begin{align}\label{unfolding_formulas}
x_i(t) = \prod_{j\in i}\dot{x}_j(t), \quad y_i(t) = \prod_{j\in i}\dot{y}_j(t), \quad \text{ and } \quad s_i = \prod_{j\in i}\dot{s}_i
\end{align}
for all $i\in I$ and $t\in\mathbb{C}$ (the terms in each product above commute, so the order of multiplication is irrelevant). By \cite[Section 2.4]{lusztig94}, we know that $w_0 = \dot{w}_0$.
\end{itemize}

We will need the following containments:
\begin{lemma}\label{folding_U_containment}
We have $U_+^{>0}\subseteq\dot{U}_+^{>0}$ and $U_-^{>0}\subseteq\dot{U}_-^{>0}$.
\end{lemma}

\begin{proof}
By \eqref{unfolding_formulas}, we have $U_+^{\ge 0} \subseteq \dot{U}_+^{\ge 0}$. Let $n\in N(T)$ denote a representative of $w_0\in W$, which is also a representative of $\dot{w}_0$ in $N(\dot{T})$ since $w_0 = \dot{w}_0$. By \eqref{Utp_definition}, we have
\[
U_+^{>0} = U_+^{\ge 0} \cap B_-nB_- \subseteq \dot{U}_+^{\ge 0} \cap \dot{B}_-n\dot{B}_- = \dot{U}_+^{>0},
\]
proving the first containment. The second containment then follows by applying $\transpose{\cdot}$.
\end{proof}

\section{Contractive flows}\label{sec:contracting}

\noindent In this section we recall some background on contractive flows, closely following \cite[Section 2]{galashin_karp_lam22a}. We then establish the topological result (\cref{flow_homeomorphism}) which we will apply to prove \cref{main}.

Let $\|\cdot\|$ denote a norm on $\mathbb{R}^N$. We say that a map $f:\mathbb{R}\times \mathbb{R}^N\to\mathbb{R}^N$ is a \define{contractive flow} if the following conditions are satisfied:
\begin{enumerate}[label=(\arabic*), leftmargin=36pt, itemsep=2pt]
\item\label{item:continuous} the map $f$ is continuous;
\item\label{item:action} for all $p\in\mathbb{R}^N$ and $t_1,t_2\in\mathbb{R}$, we have $f(0,p) = p$ and $f(t_1+t_2,p) = f(t_1, f(t_2,p))$ (that is, the flow is given by an $\mathbb{R}$-action on $\mathbb{R}^N$); and
\item\label{eq:contract} for all $p\in\mathbb{R}^N\setminus\{0\}$ and $t > 0$, we have $\|f(t,p)\| < \|p\|$.
\end{enumerate}

We have the following basic properties of contractive flows:
\begin{proposition}[{\cite[Lemma 2.2 and claim in the proof of Lemma 2.3]{galashin_karp_lam22a}}]\label{limits_act}
Let $f:\mathbb{R}\times \mathbb{R}^N\to\mathbb{R}^N$ be a contractive flow.
\begin{enumerate}[label=(\roman*), leftmargin=*, itemsep=2pt]
\item\label{limits_act_1} We have $f(t,0) = 0$ for all $t\in\mathbb{R}$.
\item\label{limits_act_2} For all $p\in\mathbb{R}^N\setminus\{0\}$, the function $t\mapsto\|f(t,p)\|$ is strictly decreasing on $(-\infty,\infty)$.
\item\label{limits_act_3} For all $p\in\mathbb{R}^N\setminus\{0\}$, we have
\[
\lim_{t\to\infty}\|f(t,p)\|=0 \quad \text{ and } \quad \lim_{t\to-\infty} \|f(t,p)\|=\infty.
\]
\item\label{limits_act_4} Let $r > 0$. Then for all $p\in\mathbb{R}^N\setminus\{0\}$, there exists a unique $t\in\mathbb{R}$ such that $\|f(t,p)\| = r$, which we denote by $t_r(p)$. Moreover, the function $t_r : \mathbb{R}^N\setminus\{0\}\to\mathbb{R}$ is continuous.\hfill\qed
\end{enumerate}
\end{proposition}

For $K\subseteq\mathbb{R}^N$ and $t\in\mathbb{R}$, we let $f(t,K)$ denote $\{f(t,p): p\in K\}$. We also let $B_r := \{p\in\mathbb{R}^N : \|p\| \le r\}$ denote the closed ball of radius $r > 0$ centered at the origin.
\begin{lemma}\label{flow_homeomorphism}
Let $K$ be a nonempty compact subset of $\mathbb{R}^N$, and let $f:\mathbb{R}\times \mathbb{R}^N\to\mathbb{R}^N$ be a contractive flow. Suppose that
\begin{align}\label{invariant}
K' := \bigcup_{t>0}f(t,K) \text{ is a relatively open subset of } K.
\end{align}
Take $r > 0$ such that $B_r\cap K \subseteq K'$. Then $K$ is homeomorphic to $B_r\cap K$.
\end{lemma}

\begin{proof}
Parts of the proof closely follow \cite[proof of Lemma 2.3]{galashin_karp_lam22a}. We begin by pointing out that $0\in K$; this follows from \eqref{invariant} and \cref{limits_act}\ref{limits_act_3}, since $K$ is nonempty and compact.

Given $p \in K \setminus \{0\}$, we define $t_\partial(p)\in(-\infty,0]$ as follows. Let $T(p):=\{t\in\mathbb{R} : f(t,p)\in K\}$. We observe:
\begin{itemize}[itemsep=2pt]
\item $0\in T(p)$;
\item if $t\in T(p)$ then $[t,\infty)\subset T(p)$, by \eqref{invariant};
\item $T(p)$ is bounded from below by \cref{limits_act}\ref{limits_act_3} and compactness of $K$; and
\item $T(p)$ is closed since it is the preimage of $K$ under the continuous map $t\mapsto f(t,p)$.
\end{itemize}
It follows that $T(p) = [t_\partial(p),\infty)$ for some $t_\partial(p)\le 0$. 

We claim that $t_\partial(p)$ is equivalently characterized as follows: for all $t\in\mathbb{R}$, we have 
\begin{align}\label{boundary_containment}
f(t,p)\in K\setminus K' \quad \Leftrightarrow \quad t = t_\partial(p).
\end{align}
To prove the backward direction, suppose otherwise that $f(t_\partial(p),p)\in K'$ (we already know that $f(t_\partial(p),p)\in K$). Then $f(t_\partial(p),p) = f(t,q)$ for some $t>0$ and $q\in K$. This implies that $f(t_\partial(p)-t,p) = q \in K$ and hence $t_\partial(p) - t \in T(p)$, contradicting the definition of $t_\partial(p)$. To prove the forward direction (via the contrapositive), we must show that given $t\neq t_\partial(p)$, we have $f(t,p)\notin K\setminus K'$. If $t < t_\partial(p)$, we have $f(t,p)\notin K$ by the definition of $t_\partial(p)$. If $t > t_\partial(p)$, then $f(t,p) = f(t-t_\partial(p),f(t_\partial(p),p))\in K'$. This proves \eqref{boundary_containment}.

We now claim that $t_\partial$ is continuous on $K\setminus\{0\}$. Given a sequence $(p_m)_{m\ge 0}$ in $K\setminus\{0\}$ converging to $p\in K\setminus\{0\}$, we must show that $t_\partial(p_m)\to t_\partial(p)$ as $m\to\infty$. Suppose otherwise that $t_\partial(p_m)\not\to t_\partial(p)$. Then by passing to a subsequence of $(p_m)_{m\ge 0}$, we may assume that $t_\partial(p_m)\to t_0$ as $m\to\infty$ for some $t_0\in[-\infty,0]$ with $t_0\neq t_\partial(p)$ (where we allow $t_0 = -\infty$). By \eqref{boundary_containment}, the sequence $(f(t_\partial(p_m),p_m))_{m\ge 0}$ is contained in $K\setminus K'$, where $K\setminus K'$ is compact (since $K$ is compact and $K'$ is relatively open in $K$).

First we assume that $t_0\neq -\infty$. Since $f$ is continuous and $K\setminus K'$ is closed, the sequence $(f(t_\partial(p_m),p_m))_{m\ge 0}$ in $K\setminus K'$ converges to $f(t_0,p)\in K\setminus K'$. Applying \eqref{boundary_containment} again we get $t_0 = t_\partial(p)$, a contradiction.

Now we instead assume that $t_0 = -\infty$. Fix $t\in\mathbb{R}$ for the moment. Since the map $q\mapsto\|f(t,q)\|$ on $\mathbb{R}^N$ is continuous at $p$, we have
\[
\|f(t,p_m)\| \ge \|f(t,p)\|-1 \quad \text{ for all $m$ sufficiently large}.
\]
Since $t_0=-\infty$, we have $t_\partial(p_m) \le t$ for all $m$ sufficiently large. By \cref{limits_act}\ref{limits_act_2}, we get
\[
\|f(t_\partial(p_m),p_m)\| \ge \|f(t,p)\|-1 \quad \text{ for all $m$ sufficiently large}.
\]
Now we take $t\to -\infty$. Since $\lim_{t\to-\infty}\|f(t,p)\| = \infty$ (by \cref{limits_act}\ref{limits_act_3}), we get $\|f(t_\partial(p_m),p_m)\|\to\infty$ as $m\to\infty$. This contradicts the fact that $\|\cdot\|$ is bounded on the compact set $K\setminus K'$. Therefore $t_\partial$ is continuous.

Now recall the continuous function $t_r : \mathbb{R}^N\setminus\{0\}\to\mathbb{R}$ from \cref{limits_act}\ref{limits_act_4}. We claim that $t_r(p) > t_\partial(p)$ for all $p\in K\setminus\{0\}$. Indeed, set $q = f(t_\partial(p),p)$, which lies in $K\setminus K'$ by \eqref{boundary_containment}. Since $B_r\cap K \subseteq K'$, we have $q\notin B_r$, so $t_r(q) > 0$. Since $t_r(q) = t_r(p) - t_\partial(p)$, we get $t_r(p) > t_\partial(p)$, as claimed. Hence by \cite[proof of Lemma 2.3 verbatim]{galashin_karp_lam22a}, the maps $\alpha:K \to B_r\cap K$ and $\beta: B_r\cap K\to K$ defined by
\[
\alpha(p) := f(t_r(p) - t_\partial(p),p) \quad \text{ and } \quad \beta(p) := f(t_\partial(p) - t_r(p),p)
\]
for all $p\neq 0$ (with $\alpha(0), \beta(0) := 0$) are inverse homeomorphisms. This completes the proof.
\end{proof}

\begin{remark}
While \cref{flow_homeomorphism} is similar to \cite[Lemma 2.3]{galashin_karp_lam22a}, there are two key differences:
\begin{itemize}[itemsep=2pt]
\item In \cite{galashin_karp_lam22a}, $K'$ (which was denoted by $Q$) is equal to the interior of $K$, which does not necessarily hold in \cref{flow_homeomorphism}. In fact, we will apply \cref{flow_homeomorphism} with $K = \overline{Z_J^{\ge 0}}$ and $K' = \bigsqcup_{J'\subseteq J}Z_{J'}^{>0}$ (see \cref{Kprime_description}), and if $J\neq\varnothing$ then $K'$ does indeed contain part of the boundary of $K$.
\item In \cite{galashin_karp_lam22a}, $B_r\cap K$ is assumed to be homeomorphic to a closed ball. While we will also apply \cref{flow_homeomorphism} when $B_r\cap K$ is homeomorphic to a closed ball, it can potentially be used to determine the homeomorphism type of spaces which are not closed balls.
\end{itemize}
\end{remark}

\begin{example}
We demonstrate that the condition \eqref{invariant} of \cref{flow_homeomorphism} is necessary to conclude that $K$ is homeomorphic to $B_r\cap K$. Consider $K\subseteq\mathbb{R}^2$ which is the union of the closed ball $B_1$ with the closed line segment from $(1,0)$ to $(2,0)$, shown below.\vspace*{-2pt}
\[
\begin{tikzpicture}[baseline=(current bounding box.center),scale=1]
\tikzstyle{out1}=[inner sep=0,minimum size=1.2mm,circle,draw=black,fill=black,semithick]
\pgfmathsetmacro{\s}{0.9};
\draw[thick,fill=teal!30](0,0)circle(1);
\draw[thick](1,0)--(2,0);
\node[out1]at(0,0)[label={[left=-1pt]\scalebox{\s}{$(0,0)$}}]{};
\node[out1]at(1,0)[label={[above right=-5pt]\scalebox{\s}{$(1,0)$}}]{};
\node[out1]at(2,0)[label={[right=-1pt]\scalebox{\s}{$(2,0)$}}]{};
\end{tikzpicture}\vspace*{-2pt}
\]
Consider the contractive flow
\[
f(t,(x,y)) := (e^{-t}x, e^{-t}y) \quad \text{ for all $t\in\mathbb{R}$ and $(x,y)\in\mathbb{R}^2$}.
\]
Then $K' := \bigcup_{t>0}f(t,K)$ is the union of the open ball $\interior(B_1)$ with the line segment from $(1,0)$ to $(2,0)$ (including the endpoint $(1,0)$ but excluding $(2,0)$). In particular, the contractive flow preserves $K$ (i.e.\ $K'\subseteq K$), and for all $0 < r < 1$, the set $B_r\cap K$ is contained in $K'$ and is homeomorphic to a closed $2$-dimensional ball. However, $K$ is not homeomorphic to a closed ball. Indeed, \cref{flow_homeomorphism} does not apply here because condition \eqref{invariant} fails, i.e., $K'$ is not relatively open in $K$ (the relative interior of $K'$ in $K$ is $K'\setminus\{(1,0)\}$).
\end{example}

\section{Flow on $\overline{G}$}\label{sec:flow}

\noindent In this section we construct a contractive flow on $\overline{G_{\ge 0}}$, which allows us to apply \cref{flow_homeomorphism} to prove our main result (\cref{main}) that $\overline{Z_J^{\ge 0}}$ is homeomorphic to a closed ball. We encourage the reader to consult \cref{sec:PSL(2)}, which exhibits the various constructions in the case of $G = \PSL_2(\mathbb{C})$.

\subsection{Algebraic preliminaries}
We set
\[
x := \sum_{i\in I}(e_i + f_i) \in \mathfrak{g},
\]
which we will use to define a contractive flow on $\overline{G}$. We begin in this subsection by establishing some properties of $\exp(tx)\in G$ which follow from \cite{lusztig94}.
\begin{lemma}\label{diagonalize}
Let $g\in G_{>0}$. Then there exist unique $h\in G$, $u_-\in U_-^{>0}$, and $s\in T_{>0}$ with the following properties:
\begin{enumerate}[label=(\roman*), leftmargin=*, itemsep=2pt]
\item\label{diagonalize_conjugate} $h^{-1}gh = s$;
\item\label{diagonalize_flag} $h\in u_-U_+$; and
\item\label{diagonalize_superpositive} $\alpha_i(s) > 1$ for all $i\in I$.
\end{enumerate}
Moreover, if $\transpose{g} = g$ then $\transpose{(h^{-1})} \in hT$.
\end{lemma}

\begin{proof}
By \cite[Corollary 8.10]{lusztig94}, there exist unique $u_-\in U_-^{>0}$, $u_+\in U_+^{>0}$, and $s\in T_{>0}$ such that $g = u_-u_+su_-^{-1}$ and $\alpha_i(s) > 1$ for all $i\in I$. We see that $u_+s$ is regular semisimple, so there exists $u_+'\in U_+$ such that $(u_+')^{-1}u_+su_+' = s$. We set $h := u_-u_+'\in G$, so that property \ref{diagonalize_flag} holds. We have
\[
h^{-1}gh = (u_+')^{-1}u_-^{-1}u_-u_+su_-^{-1}u_-u_+' = (u_+')^{-1}u_+su_+' = s,
\]
so properties \ref{diagonalize_conjugate} and \ref{diagonalize_superpositive} hold.

Now we prove uniqueness. Since $g$ is regular semisimple, the $g$-conjugates of $s$ in $T$ are precisely $W\cdot s$, and $s$ is the unique element of $W\cdot s$ satisfying property \ref{diagonalize_superpositive}. Therefore $s$ is uniquely determined. Similarly, property \ref{diagonalize_conjugate} determines $h$ up to right multiplication by $T$, whence property \ref{diagonalize_flag} uniquely determines $h$ (and also $u_-$ and $u_+'$). Finally, if $\transpose{g} = g$ then
\[
hsh^{-1} = g = \transpose{g} = \transpose{(h^{-1})}s\transpose{h},
\]
so $\transpose{(h^{-1})}\in hT$.
\end{proof}

\begin{proposition}\label{exptx_properties}~
\begin{enumerate}[label=(\roman*), leftmargin=*, itemsep=2pt]
\item\label{exptx_properties_transpose} $\transpose{\exp(tx)} = \exp(tx)$ for all $t\in\mathbb{R}$.
\item\label{exptx_properties_positivity} $\exp(tx)\in G_{>0}$ for all $t\in\mathbb{R}_{>0}$.
\item\label{exptx_properties_diagonlize} There exist $g_x\in G$, $u_x\in U_-^{>0}$, and $y\in\mathfrak{t}$ satisfying the following properties:
\begin{itemize}[itemsep=2pt]
\item $g_x^{-1}\exp(tx)g_x = \exp(ty)$ for all $t\in\mathbb{R}$;
\item $g_x \in u_xU_+$;
\item for all $t\in\mathbb{R}_{>0}$, we have $\alpha_i(\exp(ty)) > 1$ for all $i\in I$ (so $\exp(ty)\in T_{>0}$); and
\item $\transpose{(g_x^{-1})}\in g_xT$.
\end{itemize}
\end{enumerate}
\end{proposition}

\begin{proof}
\ref{exptx_properties_transpose} The anti-automorphism $\transpose{\cdot}$ on $G$ induces a linear anti-automorphism on $\mathfrak{g}$ which fixes $\mathfrak{t}$ and swaps $e_i\leftrightarrow f_i$ for all $i\in I$. This anti-automorphism on $\mathfrak{g}$ fixes $x$, so $\transpose{\cdot}$ fixes $\exp(tx)$.

\ref{exptx_properties_positivity} This is a special case of \cite[Proposition 5.9(c)]{lusztig94}.

\ref{exptx_properties_diagonlize} First let us fix $t\in\mathbb{R}_{>0}$. Then since $\exp(tx)\in G_{>0}$ and $\transpose{\exp(tx)} = \exp(tx)$, we obtain unique $g_x$, $u_x$, and $y$ with the required properties by \cref{diagonalize}. We can then check that these elements do not depend on $t$.
\end{proof}

\subsection{Construction of the contractive flow}
Fix a dominant weight $\lambda\in X(T)$ and a weight basis $v_0, v_1, \dots, v_m$ of $V_\lambda$, as in \cref{sec:Gbar}. Define
\[
A_0 := [v_0 \otimes v_0^*] \in \mathbb{P}(\End(V_\lambda)).
\]
Let $\mathbb{P}(\End(V_\lambda))_0$ denote the affine subset of $\mathbb{P}(\End(V_\lambda))$ of all elements whose coefficient of $v_0\otimes v_0^*$ is nonzero. That is, $\mathbb{P}(\End(V_\lambda))_0$ is identified with $\mathbb{C}^{(m+1)^2-1}$ as follows:
\[
\Bigg[v_0\otimes v_0^* + \sum_{\substack{0 \le i,j \le m, \\ (i,j)\neq (0,0)}}a_{i,j}v_i\otimes v_j^*\Bigg]\in\mathbb{P}(\End(V_\lambda))_0 \quad \leftrightarrow \quad (a_{i,j})_{\substack{0 \le i,j \le m, \\ (i,j)\neq (0,0)}}\in\mathbb{C}^{(m+1)^2-1}.
\]
Composing with the usual identification $\mathbb{C}\cong\mathbb{R}^2$ gives a polynomial homeomorphism $\psi_0 : \mathbb{P}(\End(V_\lambda))_0 \to \mathbb{R}^N$, where $N := 2((m+1)^2-1)$ and we endow $\mathbb{R}^N$ with the standard $2$-norm. This identifies $A_0$ with the origin of $\mathbb{R}^N$.

Let $y\in\mathfrak{t}$ be as in \cref{exptx_properties}\ref{exptx_properties_diagonlize}. We define the map $f:\mathbb{R}\times \mathbb{R}^N\to\mathbb{R}^N$ by
\begin{align}\label{flow_y}
f(t,p) := \psi_0\big((\exp(ty), \exp(ty)^{-1})\cdot\psi_0^{-1}(p)\big).
\end{align}
That is, $f$ is the flow on $\mathbb{P}(\End(V_\lambda))_0$ (identified with $\mathbb{R}^N$ via $\psi_0$) induced by the action of $(\exp(ty), \exp(ty)^{-1})\in G\times G$.
\begin{lemma}\label{Gbar_contractive_flow}
The map $f$ is a contractive flow on $\mathbb{R}^N$.
\end{lemma}

\begin{proof}
We can verify directly that $f$ satisfies conditions \ref{item:continuous} and \ref{item:action} of being a contractive flow. It remains to check condition \ref{eq:contract}. That is, given $p\in\mathbb{R}^N\setminus\{0\}$ and $t > 0$, we must show that $\|f(t,p)\| < \|p\|$.

Write
\[
\psi_0^{-1}(p) = \Bigg[v_0\otimes v_0^* + \sum_{\substack{0 \le i,j \le m, \\ (i,j)\neq (0,0)}}a_{i,j}v_i\otimes v_j^*\Bigg].
\]
Note that for all $h\in T$ and $i\in I$, we have $hv_i = \mu_i(h)v_i$ and $h^{-1}v_i^* = \mu_i(h)v_j^*$. Using this along with \eqref{GG_action_tensor}, we calculate that
\[
(\exp(ty), \exp(ty)^{-1})\cdot\psi_0^{-1}(p) = \Bigg[v_0\otimes v_0^* + \sum_{\substack{0 \le i,j \le m, \\ (i,j)\neq (0,0)}}a_{i,j}\frac{\mu_i(\exp(ty))}{\lambda(\exp(ty))}\frac{\mu_j(\exp(ty))}{\lambda(\exp(ty))}v_i\otimes v_j^*\Bigg].
\]
Now by \eqref{weight_basis_expansion}, we have
\[
\frac{\mu_j(\exp(ty))}{\lambda(\exp(ty))} = \prod_{i\in I}\alpha_i(\exp(y))^{-tc_{i,j}} \quad \text{ for all $1 \le j \le m$}.
\]
Recall that every $c_{i,j}$ is nonnegative and at least one $c_{i,j}$ is nonzero. Also, $\alpha_i(\exp(y)) > 1$ for all $i\in I$ by \cref{exptx_properties}\ref{exptx_properties_diagonlize}. Therefore $\left|\frac{\mu_j(\exp(ty))}{\lambda(\exp(ty))}\right| < 1$. This implies that
\[
\left|a_{i,j}\frac{\mu_i(\exp(ty))}{\lambda(\exp(ty))}\frac{\mu_j(\exp(ty))}{\lambda(\exp(ty))}\right| < |a_{i,j}| \quad \text{ for all $0 \le i,j \le m$ with $(i,j)\neq (0,0)$}.
\]
We obtain $\|f(t,p)\| < \|p\|$, as desired.
\end{proof}

We are really interested in the flow generated by $x$, not $y$. To this end, let $g_x$ be as in \cref{exptx_properties}\ref{exptx_properties_diagonlize}. Let $\phi$ be the automorphism of $\mathbb{P}(\End(V_\lambda))$ given by the action of $(g_x,g_x)\in G\times G$, so that $\phi^{-1}$ is the action of $(g_x^{-1}, g_x^{-1})$. Then the map
\[
\psi := \psi_0\circ\phi^{-1} : (g_x,g_x)\cdot \mathbb{P}(\End(V_\lambda))_0 \to \mathbb{R}^N
\]
identifies $(g_x,g_x)\cdot \mathbb{P}(\End(V_\lambda))_0$ with $\mathbb{R}^N$. Note that the element
\[
A_\att := \phi(A_0) = (g_x, g_x) \cdot A_0 \in (g_x,g_x)\cdot \mathbb{P}(\End(V_\lambda))_0
\]
is identified with the origin $0\in\mathbb{R}^N$ under $\psi$.
\begin{proposition}\label{flow_x}
We can write the flow $f:\mathbb{R}\times \mathbb{R}^N\to\mathbb{R}^N$ from \eqref{flow_y} as
\[
f(t,p) = \psi\big((\exp(tx), \exp(tx)^{-1})\cdot\psi^{-1}(p)\big).
\]
That is, $f$ is the flow on $(g_x,g_x)\cdot \mathbb{P}(\End(V_\lambda))_0$ (identified with $\mathbb{R}^N$ via $\psi$) induced by the action of $(\exp(tx), \exp(tx)^{-1})\in G\times G$.
\end{proposition}

\begin{proof}
This follows by a direct calculation using the fact that \[
(g_x^{-1},g_x^{-1})(\exp(tx), \exp(tx)^{-1})(g_x,g_x) = (\exp(ty),\exp(ty)^{-1}) \quad \text{ in } G\times G,
\]
from \cref{exptx_properties}\ref{exptx_properties_diagonlize}.
\end{proof}

The next step is to show that $\overline{G_{\ge 0}}$ is identified with a nonempty compact subset of $\mathbb{R}^N$. To this end, for the remainder of this section we fix $\lambda\in X(T)$ to be a dominant regular weight. (We mention that $A_0$ is equal to the point $z_\varnothing^\circ\in Z_\varnothing$.)

\begin{lemma}\label{nonnegative_to_positive}
Let $g_1\in U_+^{>0}$ and $g_2\in U_-^{>0}$. Then
\[
(g_1,g_2^{-1})\cdot \overline{G_{\ge 0}} \subseteq \mathbb{P}(\End(V_\lambda))_0.
\]
\end{lemma}

We first prove \cref{nonnegative_to_positive} in the simply-laced case, and then prove it in general via folding.
\begin{proof}[Proof of \cref{nonnegative_to_positive} in the simply-laced case]
We adopt the setup of \cref{sec:Gbar}, and additionally assume that $v_0, \dots, v_m$ is the canonical basis of $V_\lambda$ (where $G$ is simply-laced). Let $\mathbb{P}(\End(V_\lambda))_{\ge 0}$ denote the subset of $\mathbb{P}(\End(V_\lambda))$ of elements of the form $[\sum_{0 \le i,j \le m}a_{i,j}v_i\otimes v_j^*]$ such that $a_{i,j} \ge 0$ for all $0 \le i,j \le m$. By \cite[Proposition 3.2(a)]{lusztig94} (applied to the simply connected cover of $G$), we have $\overline{G_{\ge 0}} \subseteq \mathbb{P}(\End(V_\lambda))_{\ge 0}$. Thus it suffices to show that
\[
(g_1,g_2^{-1})\cdot \mathbb{P}(\End(V_\lambda))_{\ge 0} \subseteq \mathbb{P}(\End(V_\lambda))_0.
\]

We do so by showing that for all $0 \le i,j \le m$, the coefficient of $v_0\otimes v_0^*$ in the basis expansion of
\[
(g_1,g_2^{-1})\cdot v_i\otimes v_j^*
\]
is positive. It suffices to check that the coefficient of $v_0$ in $g_1v_i$ is positive and the coefficient of $v_0^*$ in $g_2^{-1}v_j^*$ is positive. This follows from \cite[Lemma 5.10 and proof of Proposition 5.9]{barkley_karp}, since every $e_i$ (for $i\in I$) acts on the canonical basis with nonnegative structure constants.
\end{proof}

\begin{proof}[Proof of \cref{nonnegative_to_positive} in general, via folding]
We adopt the setup of \cref{sec:folding}. There exists a dominant weight $\dot{\lambda}\in X(\dot{T})$ such that $\lambda = \dot{\lambda}|_T\in X(T)$, so that we may regard $V_\lambda$ as a subspace of $V_{\dot{\lambda}}$. Since the highest weight spaces of $V_\lambda$ and $V_{\dot{\lambda}}$ are both spanned by $v_0$, we have
\[
\mathbb{P}(\End(V_\lambda))_0 = \mathbb{P}(\End(V_{\dot{\lambda}}))_0\cap\mathbb{P}(\End(V_\lambda)).
\]
Also, since $V_\lambda \subseteq V_{\dot{\lambda}}$, the Zariski closure of $G$ in $\mathbb{P}(\End(V_{\dot{\lambda}}))$ is the wonderful compactification of $G$ as a consequence of \cite[Lemma in Section 4]{de_concini_procesi83}.  Therefore it suffices to show that
\begin{align}\label{folding_to_prove}
(g_1,g_2^{-1})\cdot \overline{G_{\ge 0}} \subseteq \mathbb{P}(\End(V_{\dot{\lambda}}))_0.
\end{align}
By \cref{folding_U_containment}, we have $g_1\in\dot{U}_+^{>0}$ and $g_2\in\dot{U}_-^{>0}$. From the definitions, we have $G_{\ge 0} \subseteq \dot{G}_{\ge 0}$, so $\overline{G_{\ge 0}} \subseteq \overline{\dot{G}_{\ge 0}}$. Thus \eqref{folding_to_prove} follows from \cref{nonnegative_to_positive} applied to $\dot{G}$ (which we have already proved).
\end{proof}

\begin{lemma}\label{positive_nonzero_coefficient}
We have $\overline{G_{\ge 0}} \subseteq (g_x, g_x)\cdot \mathbb{P}(\End(V_\lambda))_0$.
\end{lemma}

\begin{proof}
Equivalently, we must show that $(g_x^{-1}, g_x^{-1})\cdot \overline{G_{\ge 0}} \subseteq \mathbb{P}(\End(V_\lambda))_0$. To this end, by \cref{exptx_properties}\ref{exptx_properties_diagonlize} we can write $g_x = u_xu_+$ for some $u_+\in U_+$ and $\transpose{(g_x^{-1})} = g_xs$ for some $s\in T$, whence
\[
g_x^{-1} = s\transposeminus{g_x} = s\transposeminus{u_+}\transposeminus{u_x}.
\]
Therefore
\[
(g_x^{-1}, g_x^{-1}) = (s\transposeminus{u_+}, u_+^{-1})(\transposeminus{u_x},u_x^{-1}).
\]
Note that $(s\transposeminus{u_+}, u_+^{-1})\in B_-\times B_+$, and $B_-\times B_+$ preserves $\mathbb{P}(\End(V_\lambda))_0$. Thus it suffices to show that $(\transposeminus{u_x},u_x^{-1})\cdot \overline{G_{\ge 0}} \subseteq \mathbb{P}(\End(V_\lambda))_0$. This follows from \cref{nonnegative_to_positive} by taking $g_1 = \transposeminus{u_x} \in U_+^{>0}$ and $g_2 = u_x \in U_-^{>0}$.
\end{proof}

\subsection{Topological preliminaries}
Recall from \cref{sec:Gbar} that $\overline{Z_J^{\ge 0}}$ (for $J\subseteq I$) denotes the Euclidean closure of $Z_J^{\ge 0}$ (or $Z_J^{>0}$, by \cref{he_ZJ}\ref{he_ZJ_containment}) inside $\overline{G_{\ge 0}}$, and that $\overline{Z_J}$ denotes the Zariski closure of $Z_J$. In this subsection we establish several topological properties of $\overline{Z_J^{\ge 0}}$ which are required in order to apply \cref{flow_homeomorphism} to $\overline{Z_J^{\ge 0}}$.
\begin{proposition}\label{stratum_closure_description}
For all $J\subseteq I$, we have
\[
\overline{Z_J^{\ge 0}} = \bigsqcup_{J'\subseteq J}Z_{J'}^{\ge 0} = \overline{Z_J}\cap \overline{G_{\ge 0}}.
\]
\end{proposition}

\begin{proof}
Let $J\subseteq I$. By \cref{ZJ_closure}, we obtain both the second equality and the $\subseteq$ containment of the first equality. It remains to prove the $\supseteq$ containment, i.e., that $Z_{J'}^{\ge 0}\subseteq \overline{Z_J^{\ge 0}}$ for all $J'\subseteq J$. Since $T_{>0}$ is contained in the Euclidean closure of $G_{>0}$, we have $(s,e)\cdot z_J^\circ \in \overline{Z_J^{\ge 0}}$ for all $s\in T_{>0}$. Note that the limit of $(s,e)\cdot z_J^\circ$ equals $z_{J'}^\circ$, where $\alpha_i(s) = 1$ for all $i\in J'$ and $\alpha_i(g)\to\infty$ for all $i\notin J'$. In particular, we have $z_{J'}^\circ\in \overline{Z_J^{\ge 0}}$. This implies that $Z_{J'}^{\ge 0} \subseteq \overline{Z_J^{\ge 0}}$ by \cref{he_ZJ}\ref{he_ZJ_containment}, as desired.
\end{proof}

\begin{lemma}\label{Kprime_description_lemma}
For all $J\subseteq I$, we have
\[
\bigcup_{t>0}(\exp(tx), \exp(tx)^{-1})\cdot Z_J^{\ge 0} = Z_J^{>0}.
\]
\end{lemma}

\begin{proof}
The $\subseteq$ containment follows from \cref{he_ZJ}\ref{he_ZJ_containment}, since $\exp(tx)\in G_{>0}$ by \cref{exptx_properties}\ref{exptx_properties_positivity}. For the $\supseteq$ containment, let $A\in Z_J^{>0}$. Since $Z_J^{>0}$ is open in the real part of $Z_J$, there exists $t>0$ such that $(\exp(tx)^{-1},\exp(tx))\cdot A\in Z_J^{>0}$. Then $A\in (\exp(tx)^{-1},\exp(tx))\cdot Z_J^{\ge 0}$, as desired.
\end{proof}

Recall from \cref{sec:Gbar} that $v_0$ is the unique basis vector of $V_\lambda$ of highest weight $\lambda$, where $\lambda$ is a dominant regular weight. We assume that $I = \{1, \dots, \ell\}$, and we reorder the remaining basis vectors of $V_\lambda$ so that $v_i$ is the unique basis vector of weight $\lambda - \alpha_i$ for all $i\in [\ell]$.

We now recall some background from \cite[Sections 2.2 and 2.4]{evens_jones08}. Let $Y$ denote the Zariski closure of $T$ inside $\mathbb{P}(\End(V_\lambda))_0$. Then we have the isomorphism $F : \mathbb{C}^\ell\to Y$ given by
\begin{align}\label{torus_lambda_isomorphism}
F(y_1, \dots, y_r) := \Bigg[v_0\otimes v_0^* + \sum_{i=1}^ry_iv_i\otimes v_i^* + \sum_{i=\ell+1}^m\bigg(\prod_{i'=1}^\ell y_i^{c_{i',i}}\bigg)v_i\otimes v_i^*\Bigg].
\end{align}
For $J\subseteq I$, we let:
\begin{itemize}[itemsep=2pt]
\item $Y_J \cong \mathbb{C}^{|J|}$ denote the set of all points $F(y_1, \dots, y_\ell)$ such that $y_i = 0$ for all $i\notin J$ (so that $Y_I = Y$); and
\item $Y_J^\circ \cong (\mathbb{C}^\times)^{|J|}$ denote the set of all points $F(y_1, \dots, y_\ell)$ such that $y_i\neq 0$ for all $i\in J$ and $y_i = 0$ for all $i\notin J$.
\end{itemize}
(We remark that in \cite{evens_jones08}, $Y_J$ is denoted by $Z_{I\setminus J}$.) In particular, we have
\begin{align}\label{YJcirc_decomposition}
Y_J = \bigsqcup_{J'\subseteq J}Y_{J'}^\circ.
\end{align}

Note that $z_J^\circ = F(e_J)$, where $e_J\in\mathbb{C}^\ell$ is the unique point such that $(e_J)_i = 1$ for all $i\in J$ and $(e_J)_i = 0$ for all $i\notin J$. Also, the $Y_J^\circ$'s (for $J\subseteq I$) are precisely the $T\times T$ (or $T\times\{e\}$) orbits of $Y$; explicitly, $Y_J^\circ = (T,e)\cdot z_J^\circ$. We also observe that $Y_J$ is an affine open subset of the toric variety $\overline{Z_J}\cap\overline{T}$.

\begin{lemma}\label{Taff_intersection}
Let $J\subseteq I$.
\begin{enumerate}[label=(\roman*), leftmargin=*, itemsep=2pt]
\item\label{Taff_intersection_closed} The intersection $Y_J \cap \overline{T_{>0}}\cong \mathbb{R}_{\ge 0}^{|J|}$ consists of all points $F(y_1, \dots, y_\ell)$ such that $y_i\ge 0$ for all $i\in J$ and $y_i = 0$ for all $i\notin J$.
\item\label{Taff_intersection_open} We have $(T_{>0},e)\cdot z_J^\circ = Y_J^\circ \cap \overline{T_{>0}} \cong \mathbb{R}_{>0}^{|J|}$; both sets consist of all points $F(y_1, \dots, y_\ell)$ such that $y_i > 0$ for all $i\in J$ and $y_i = 0$ for all $i\notin J$.
\end{enumerate}
\end{lemma}

\begin{proof}
This follows from the definitions and the properties of $Y_J$ above, using the fact that $Y \cap \overline{T_{>0}}$ is the image under $F$ of $\mathbb{R}_{\ge 0}^\ell$.
\end{proof}

\subsection{Proof of \texorpdfstring{\cref{main}}{the main theorem}}\label{sec:flow_proof}
For the remainder of this subsection, fix $J\subseteq I$, and set $K := \psi(\overline{Z_J^{\ge 0}}) \subseteq \mathbb{R}^N$, which is well-defined and homeomorphic to $\overline{Z_J^{\ge 0}}$ by \cref{positive_nonzero_coefficient}. Set
\[
K' := \bigcup_{t>0}f(t,K),
\]
as in \eqref{invariant}. Our goal is to apply \cref{flow_homeomorphism}.

\begin{lemma}\label{Kprime_description}
We have
\[
\psi^{-1}(K') = \bigcup_{t>0}(\exp(tx), \exp(tx)^{-1})\cdot \overline{Z_J^{\ge 0}} = \bigsqcup_{J'\subseteq J}Z_{J'}^{>0}.
\]
\end{lemma}

\begin{proof}
The first equality follows from \cref{flow_x}. The second equality follows by combining \cref{stratum_closure_description,Kprime_description_lemma}.
\end{proof}

\begin{lemma}\label{Kprime_relatively_open}
The set $K'$ is a relatively open subset of $K$.
\end{lemma}

\begin{proof}
By \cref{Kprime_description}, we must equivalently show that $\bigsqcup_{J'\subseteq J}Z_{J'}^{>0}$ is relatively open in $\overline{Z_J^{\ge 0}}$ (it is a subset of $\overline{Z_J^{\ge 0}}$ by \cref{stratum_closure_description}). By \cref{he_ZJ}\ref{he_ZJ_paramtrization}, we have
\begin{align}\label{Tbar_reduction}
\bigsqcup_{J'\subseteq J}Z_{J'}^{>0} \cong U_-^{>0} \times \bigsqcup_{J'\subseteq J}(T_{>0},e)\cdot z_{J'}^\circ \times U_+^{>0}.
\end{align}
Since $U_-^{>0}$ is open in the real part of $U_-$ (and similarly for $U_+$), it suffices to show that $\bigsqcup_{J'\subseteq J}(T_{>0},e)\cdot z_{J'}^\circ$ is relatively open in $\overline{Z_J} \cap \overline{T_{>0}}$. By \cref{Taff_intersection} and \eqref{YJcirc_decomposition}, we have
\[
\bigsqcup_{J'\subseteq J}(T_{>0},e)\cdot z_{J'}^\circ = \bigsqcup_{J'\subseteq J}Y_{J'}^\circ\cap\overline{T_{>0}} = Y_J\cap\overline{T_{>0}} \cong \mathbb{R}_{\ge 0}^{|J|}.
\]
This is relatively open in $\overline{Z_J} \cap \overline{T_{>0}}$ since it equals the intersection of $\overline{Z_J} \cap \overline{T_{>0}}$ with the open subset $Y_J$.
\end{proof}

\begin{remark}
We explain an alternative way to prove \cref{Kprime_relatively_open}, which involves more machinery but is perhaps more intuitive (we omit the technical details). The intersection $\overline{Z_J}\cap\overline{T}$ is a toric variety whose nonnegative part is $\overline{Z_J}\cap\overline{T_{>0}}$. It is a general theorem of toric varieties (see \cite[Section 4.2]{fulton93} or \cite[Section 8]{sottile03a}) that the nonnegative part of a toric variety is homeomorphic to its moment polytope $P$ via the algebraic moment map $\mu$. In our case, $P$ is the \emph{$W$-permutohedron}, i.e., the convex hull of $W\cdot\lambda$ in $\mathfrak{t}^*$. The $T$-fixed point $z_\varnothing^\circ$ corresponds to a vertex $\mu(z_\varnothing^\circ)$ of $P$, and $\mu(Y_J\cap\overline{T_{>0}})$ is the star of this vertex, i.e., the union of all open faces whose closure contains $z_\varnothing^\circ$. This is a relatively open subset of $P$ since it is obtained from $P$ by deleting all closed facets which do not contain $z_\varnothing^\circ$, thereby proving \cref{Kprime_relatively_open}.
\end{remark}
\begin{figure}[tb]
\begin{center}
\[
\begin{tikzpicture}[baseline=(current bounding box.center)]
\tikzstyle{vertexu}=[inner sep=0,minimum size=1.8mm,circle,draw=gray,fill=gray!30,semithick]
\tikzstyle{vertexs}=[inner sep=0,minimum size=1.8mm,circle,draw=black,fill=teal!80,semithick]
\pgfmathsetmacro{\r}{1.68};
\pgfmathsetmacro{\s}{1.0};
\pgfmathsetmacro{\hs}{0.80};
\pgfmathsetmacro{\vs}{0.40};
\useasboundingbox(0,0)circle(\r+0.24);
\fill[color=black!20!teal!50](-90:\r)--(-30:\r)--(30:\r)--(90:\r)--(150:\r)--(210:\r)--cycle;
\draw[very thick,color=black!20!teal!80](90:\r)--(150:\r) (90:\r)--(30:\r);
\draw[very thick, color=gray!30](-210:\r)--(-150:\r)--(-90:\r)--(-30:\r)--(30:\r);
\node[vertexu](123)at(-90:\r){};
\node[vertexu](132)at(-30:\r){};
\node[vertexu](231)at(30:\r){};
\node[vertexs](321)at(90:\r){};
\node[vertexu](312)at(150:\r){};
\node[vertexu](213)at(210:\r){};
\node[inner sep=0]at($(321)+(0,\vs)$){\scalebox{\s}{$\mu(z_\varnothing^\circ)$}};
\end{tikzpicture}\vspace*{-8pt}
\]
\caption{The moment polytope $P$ of $\overline{T}$ for $G = \PSL_3(\mathbb{C})$, which is the permutohedron of the symmetric group $W = \symgp{3}$. The star of the vertex $\mu(z_\varnothing^\circ)$ is shaded; we see it is relatively open in $P$.}
\label{fig:hexagon}
\end{center}
\end{figure}
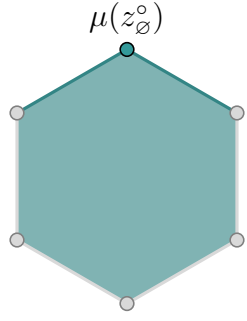

We are now ready to apply \cref{flow_homeomorphism} to show that $\overline{Z_J^{\ge 0}}$ is homeomorphic to a closed ball. By \cref{Gbar_contractive_flow}, $f$ is a contractive flow on $\mathbb{R}^N$, and \eqref{invariant} holds by \cref{Kprime_relatively_open}. It remains to show that there exists $r > 0$ such that $B_r\cap K$ is contained in $K'$ and is homeomorphic to a closed ball. Applying $\psi^{-1}$ (using \cref{Kprime_description}), it suffices to show that there exists $r > 0$ such that $\psi^{-1}(B_r)\cap\overline{Z_J^{\ge 0}}$ is contained in $\bigsqcup_{J'\subseteq J}Z_{J'}^{>0}$ and is homeomorphic to a closed ball.

We again use \eqref{Tbar_reduction} and proceed as in the proof of \cref{Kprime_relatively_open}. Note that $\psi^{-1}(0) = A_\att = (g_x,g_x) \cdot z_\varnothing^\circ$ (which is an element of \eqref{Tbar_reduction}), and $U_-^{>0}$ and $U_+^{>0}$ are both homeomorphic to the open ball $\mathbb{R}_{>0}^d$, where $d$ is the length of $w_0$. Hence it suffices to show that there exists $r > 0$ such that $\psi^{-1}(B_r) \cap \overline{Z_J} \cap \overline{T_{>0}}$ is contained in $Y_J\cap\overline{T_{>0}}$ and is homeomorphic to a closed ball. This follows from the fact that $Y_J \cap \overline{T_{>0}}$ is relatively open in $\overline{Z_J} \cap \overline{T_{>0}}$ and homeomorphic to $\mathbb{R}_{\ge 0}^{|J|}$; intersecting with $\psi^{-1}(B_r)$ is equivalent in $\mathbb{R}_{\ge 0}^{|J|}$ to taking the $r$-sublevel set (i.e.\ the preimage of $[0,r]$) of a proper continuous nonnegative semialgebraic function on $\mathbb{R}^{|J|}$ whose preimage of $0$ is the origin. It is a general theorem of semialgebraic geometry (see \cite[proof of Theorem 9.3.6]{bochnak_coste_roy98} and \cite[Proposition 1]{coste_kurdyka92}) that the topology of such a sublevel set does not depend on the choice of such a function (for all $r > 0$ sufficiently small); therefore such a sublevel set is homeomorphic to a closed ball for all $r > 0$ sufficiently small. This completes the proof of \cref{main}.\hfill\qed

\begin{remark}
The \define{complete flag variety} of $G$ is the homogeneous space $G/B_+$. Its totally nonnegative part $(G/B_+)_{\ge 0}$ (introduced by Lusztig \cite[Section 8]{lusztig94}) is the Euclidean closure of the image of $U_-^{>0}$. Alternatively, by applying the map $\transpose{\cdot}$, we may regard the complete flag variety to be $G/B_-$ and its totally nonnegative part (denoted $(G/B_-)_{\ge 0}$) to be the Euclidean closure of the image of $U_+^{>0}$. It was shown in \cite{galashin_karp_lam19} that $(G/B_+)_{\ge 0}$ is homeomorphic to a closed ball using a contractive flow, and more generally in \cite{galashin_karp_lam22} that $(G/B_+)_{\ge 0}$ is a regular CW complex (using entirely different methods; see \cref{he_xie_remark}).

When $J = \varnothing$, the stratum $Z_J$ is isomorphic to $G/B_+ \times G/B_-$ \cite[Proposition 2.25]{evens_jones08}. Its totally nonnegative part is $Z_\varnothing^{\ge 0} = \overline{Z_\varnothing^{\ge 0}}$ (the equality follows from \cref{stratum_closure_description}), which is homeomorphic to the product $(G/B_+)_{\ge 0} \times (G/B_-)_{\ge 0}$ of two totally nonnegative flag varieties. (This follows from the definitions using the fact that $Z_\varnothing^{\ge 0}$ is the Euclidean closure of $Z_\varnothing^{>0}$, and $Z_\varnothing^{>0}\cong U_-^{>0}\times U_+^{>0}$ by \cref{he_ZJ}\ref{he_ZJ_paramtrization}.) In particular, \cref{main} recovers the fact that $(G/B_+)_{\ge 0} \times (G/B_-)_{\ge 0}$ is homeomorphic to a closed ball.
\end{remark}

\subsection{Example: \texorpdfstring{$\PSL_2(\mathbb{C})$}{PSL(2)}}\label{sec:PSL(2)}
\begin{figure}[tb]
\begin{center}~\vspace*{-12pt}
\[
\includegraphics[scale=0.35]{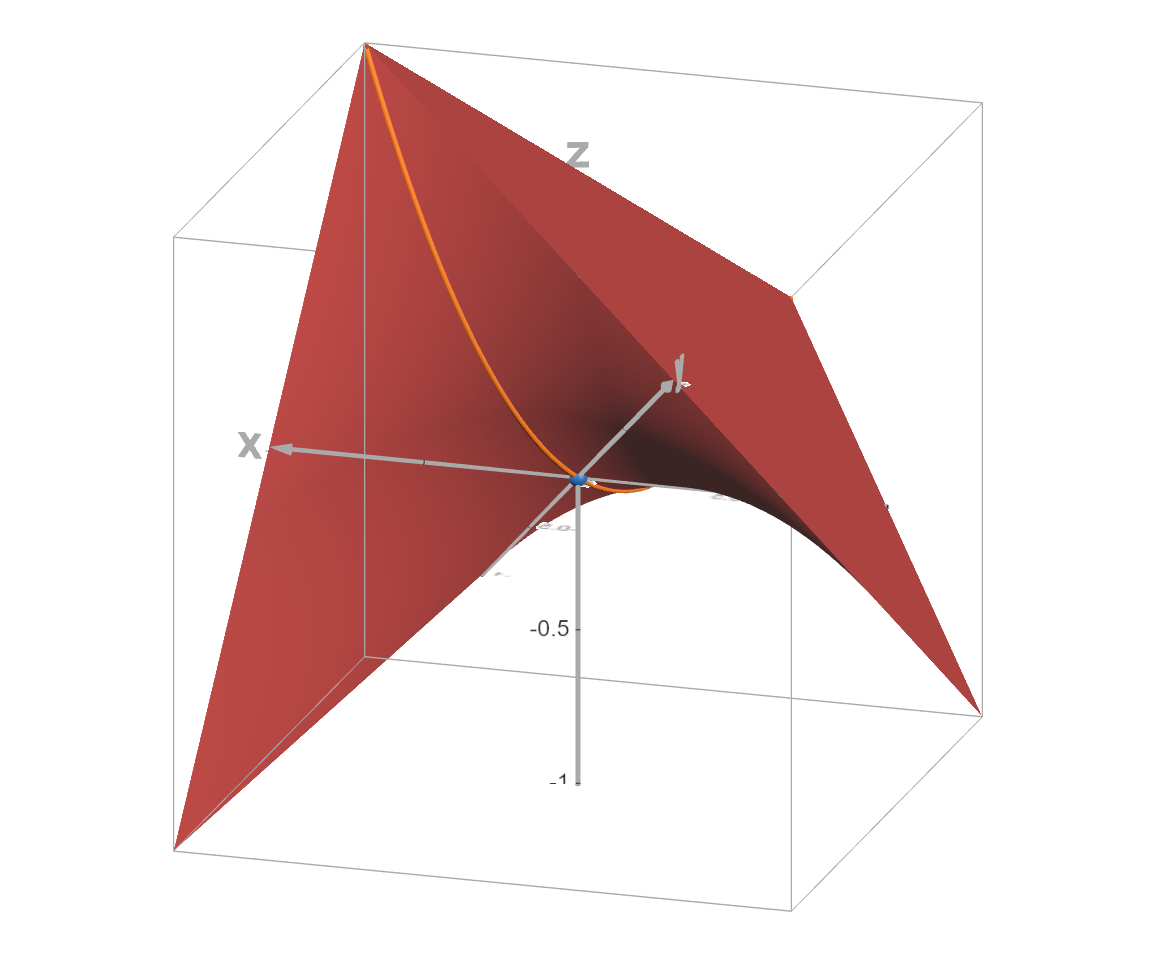}\vspace*{-18pt}
\]
\caption{A plot of $\overline{G_{\ge 0}}$ for $G = \PSL_2(\mathbb{C})$ obtained using \href{https://www.desmos.com/}{\texttt{Desmos}}. The attractor for the contractive flow is the origin (depicted in blue), with two trajectories depicted in orange.}
\label{fig:PSL(2)}
\end{center}
\end{figure}
We illustrate \cref{main} and its proof in the case of $G = \PSL_2(\mathbb{C})$. Recall from \cref{eg:pinning} that $G$ is the space of $2\times 2$ matrices with determinant $1$ modulo rescaling by $\pm 1$. Let $V = \mathbb{C}^2$, so that we can identify $\End(V)$ with the space of $2\times 2$ matrices. Then we can realize $\overline{G}$ in $\mathbb{P}(\End(V))$ as the space of nonzero $2\times 2$ matrices modulo rescaling:
\[
\overline{G} = \left\{[A] : A = \begin{pmatrix}a & b \\ c & d\end{pmatrix} \text{ such that } a,b,c,d \text{ not all } 0\right\}.
\]
(Technically $V$ is not a representation of $G$ as in \cref{sec:Gbar}, but rather of the simply connected cover $\SL_2(\mathbb{C})$. We still obtain a valid realization of the wonderful compactification of $\PSL_2(\mathbb{C})$, and it is the most convenient one for doing calculations.)

The $G\times G$-action on $\overline{G}$ is given by $(g_1,g_2)\cdot [A] = [g_1Ag_2^{-1}]$. We calculate that
\[
\overline{G} = Z_I \sqcup Z_\varnothing,
\]
where $Z_I = G$ and $Z_\varnothing$ is the group of $2\times 2$ matrices of rank $1$ modulo rescaling. We also have
\[
\overline{G_{\ge 0}} = \left\{[A] : A = \begin{pmatrix}a & b \\ c & d\end{pmatrix}\neq 0 \text{ such that } a,b,c,d,ad-bc \ge 0\right\},
\]
and $\overline{Z_\varnothing^{\ge 0}}$ is the subset of matrix classes of rank $1$.

The contractive flow on $\overline{G_{\ge 0}}$ is given by the action of $(\exp(tx), \exp(tx)^{-1})\in G\times G$, where
\[
x = e_1 + f_1 = \begin{pmatrix}
0 & 1 \\
1 & 0
\end{pmatrix} \quad \text{ and } \quad \exp(tx) = \begin{pmatrix}
\cosh(t) & \sinh(t) \\
\sinh(t) & \cosh(t)
\end{pmatrix}.
\]
Explicitly, $(\exp(tx), \exp(tx)^{-1}) \cdot [A] = [\exp(tx)A\exp(tx)]$. However, it is more convenient to work in coordinates $(\alpha,\beta,\gamma)\in\mathbb{R}^3$ coming from \cref{flow_x}. Define the eigenvectors
\begin{align}\label{PSL(2)_eigenvectors}
w_0 := \begin{pmatrix}1 \\ 1\end{pmatrix}\!,\, w_1 := \begin{pmatrix}1 \\ -1\end{pmatrix}\!, \quad \text{ so that } \quad \exp(tx)w_0 = e^tw_0,\, \exp(tx)w_1 = e^{-t}w_1.
\end{align}
Note that the $4$ matrices $w_i\transposeminus{w_j}$ (for $i,j\in\{0,1\}$) form a basis for the space of $2\times 2$ matrices; we identify $\mathbb{R}^3$ with the space of real matrices (modulo rescaling) $[A]$ whose coefficient of $w_0\transposeminus{w_0}$ is nonzero:
\begin{align*}
A &= \begin{pmatrix}1 & 1 \\ 1 & 1\end{pmatrix} + \alpha\begin{pmatrix}1 & -1 \\ 1 & -1\end{pmatrix} + \beta\begin{pmatrix}1 & 1 \\ -1 & -1\end{pmatrix} + \gamma\begin{pmatrix}1 & -1 \\ -1 & 1\end{pmatrix} \\
&= \begin{pmatrix}
1 + \alpha + \beta + \gamma & 1 - \alpha + \beta - \gamma \\
1 + \alpha - \beta - \gamma & 1 - \alpha - \beta + \gamma
\end{pmatrix}.
\end{align*}
Equivalently, this is the space of real $[A]$ such that the sum of the entries of $A$ is nonzero, which in particular includes all of $\overline{G_{\ge 0}}$ (in agreement with \cref{positive_nonzero_coefficient}). This embeds $\overline{G_{\ge 0}}$ inside $\mathbb{R}^3$ as the subset
\begin{align}\label{R3_inequalities}
1 + \alpha + \beta + \gamma, 1 - \alpha + \beta - \gamma, 1 + \alpha - \beta - \gamma, 1 - \alpha - \beta + \gamma, 4\gamma - 4\alpha\beta \ge 0,
\end{align}
as depicted in \cref{fig:PSL(2)}. We see that there are three $2$-faces, five edges, and four vertices (the cell decomposition is described in \cite[Section 6.2]{he_xie}).

Using \eqref{PSL(2)_eigenvectors}, we calculate how the contractive flow $(\exp(tx), \exp(tx)^{-1}) \cdot [A]$ acts on $\mathbb{R}^3$:
\[
f(t) \cdot (\alpha,\beta,\gamma) = (e^{-2t}\alpha, e^{-2t}\beta, e^{-4t}\gamma).
\]
In particular, $f(t)\cdot(\alpha,\beta,\gamma)$ converges as $t\to\infty$ to $(0,0,0)$, which corresponds to the matrix class
\[
\left[\begin{pmatrix}
1 & 1 \\
1 & 1
\end{pmatrix}\right] = z_\varnothing^\circ.
\]
In \cref{fig:PSL(2)} we have depicted the trajectories from the vertices \scalebox{0.8}{$\left[\begin{pmatrix}4 & 0 \\ 0 & 0\end{pmatrix}\right]$} and \scalebox{0.8}{$\left[\begin{pmatrix}0 & 0 \\ 0 & 4\end{pmatrix}\right]$} to the attractor. We note that $\overline{Z_\varnothing^{\ge 0}}$ is obtained from \eqref{R3_inequalities} by imposing the additional constraint $4\gamma - 4\alpha\beta = 0$, corresponding to the (curved) quadrilateral face in \cref{fig:PSL(2)}.

\bibliographystyle{alpha}
\bibliography{ref}

\end{document}